\documentclass[letterpaper, 10 pt, conference]{eccconf}  % Comment this line out if you need a4paper

\IEEEoverridecommandlockouts                              % This command is only needed if 
\usepackage{amsmath} % assumes amsmath package installed
\usepackage{amssymb}  % assumes amsmath package installed 

\usepackage{cleveref}
\usepackage{todonotes}
\usepackage{textgreek}
\usepackage{tikz}
\usepackage{pgfplots}
\pgfplotsset{compat=1.18}

\newlength{\figurewidth}
\newtheorem{theorem}{Theorem}
\newtheorem{definition}{Definition}
\newtheorem{proposition}{Proposition}

\newtheorem{assumption}{Assumption}
\newtheorem{remark}{Remark}

\crefname{figure}{fig.}{figs.}
\Crefname{figure}{Fig.}{Figs.}

\title{\LARGE \bf
Sum-of-Squares Stability Verification on Manifolds with Applications in Spacecraft Attitude Control
}

\author{Fabian Geyer, Friedrich Tuttas, Walter Fichter, and Torbjørn Cunis
\thanks{This is an extended version of a paper accepted at the European Control Conference (ECC) [2026]. This work was funded by the Deutsche Forschungsgemeinschaft project number 516238647 -- SFB1667/1 (ATLAS -- Advancing Technologies for Low-Altitude Satellites).}
\thanks{Institute of Flight Mechanics and Controls, University of Stuttgart, 70569 Stuttgart, Germany. Corresponding author: {\tt\small fabian.geyer@ifr.uni-stuttgart.de}}%
}

\newcommand{\R}{\mathbb{R}}
\newcommand{\setH}{\mathcal{H}}
\newcommand{\om}{\omega_\mathrm{BI}^{\mathrm{B}}}
\newcommand{\dotom}{\dot{\omega}_\mathrm{BI}^{\mathrm{B}}}
\newcommand{\w}{{s}_\mathrm{BI}^{\mathrm{B}}}
\newcommand{\wdot}{\dot{s}_\mathrm{BI}^{\mathrm{B}}}

\newcommand{\sos}[1]{\Sigma_{#1}[x]}
\newcommand{\realpoly}[1]{\mathbb{R}_{#1}[x]}

\newcommand{\skewop}[1]{\left[ {#1} \times \right]}

\newcommand{\omBO}{\omega_\mathrm{BO}^{\mathrm{B}}}
\newcommand{\omBOdot}{\dot{\omega}_\mathrm{BO}^{\mathrm{B}}}
\newcommand{\xBO}{\hat{x}_{\mathrm{O}}^{\mathrm{B}}}
\newcommand{\yBO}{\hat{y}_{\mathrm{O}}^{\mathrm{B}}}
\newcommand{\zBO}{\hat{z}_{\mathrm{O}}^{\mathrm{B}}}

\newcommand{\qBO}{q_\mathrm{BO}}
\newcommand{\qBOdot}{\dot{q}_\mathrm{BO}}

\newcommand{\cosd}{C_{\delta}}

\newcommand{\inner}[2]{\left< #1, #2 \right>}

\newcommand{\norm}[1]{\left\| #1 \right\|_2}

\begin{document}

\bstctlcite{IEEEtranBSTcontrol}

\maketitle
\thispagestyle{empty}
\pagestyle{empty}

%%%%%%%%%%%%%%%%%%%%%%%%%%%%%%%%%%%%%%%%%%%%%%%%%%%%%%%%%%%%%%%%%%%%%%%%%%%%%%%%
\begin{abstract}
In the context of spacecraft attitude control, parametrizations such as direction vectors or quaternions are often used to avoid singularities in the attitude representation. This, however, complicates the stability analysis of the system since, given the additional unit constraints, the resulting dynamics evolve on non-contractible manifolds. In this paper, we present a framework to verify almost global asymptotic stability of such systems using LaSalle's invariance principle and sum-of-squares programming, simplifying the search for Lyapunov functions. The framework is then applied to two examples: two-axis attitude acquisition utilizing aerodynamics in very low Earth orbits, and three-axis attitude acquisition for a satellite subject to gravity gradient torques in a circular orbit.
\end{abstract}

%%%%%%%%%%%%%%%%%%%%%%%%%%%%%%%%%%%%%%%%%%%%%%%%%%%%%%%%%%%%%%%%%%%%%%%%%%%%%%%%
\section{Introduction}
\label{sec:intro}
Stability certification for complex, autonomous nonlinear systems is of critical importance in many engineering applications, including aerospace systems.
For spacecraft, the ability to formally certify the stability and robustness of the attitude control system is paramount for mission success.
A central problem in this domain is the mathematical representation of attitude. Since common representations such as Euler angles or Modified Rodrigues Parameters suffer from singularities, constrained representations such as quaternions or direction vectors are often preferred \cite{wieQuaternionFeedbackRegulator1989, fichterPrinciplesSpacecraftControl2023}.
However, these advantageous representations introduce their own analytical challenges.
Due to the additional unit constraints, the resulting dynamics evolve on non-contractible manifolds (e.g., quaternions evolve on the unit 3-sphere) complicating system analysis. 
A typical approach to verify stability for such nonlinear systems is the analytical construction of suitable Lyapunov functions for the nonlinear system. This has been successfully demonstrated for certain classes of quaternion-based controllers, as exemplified by \cite{chaturvediAlmostGlobalAttitude2006}, which proved almost global asymptotic stability (AGAS) for spacecraft in circular orbits.
Unfortunately, these manual, pen-and-paper methods are difficult to apply and become practically impossible as system complexity increases.
Consider, for example, a satellite subject to aerodynamic torques in very low Earth orbits (VLEO) \cite{crispSOARSatelliteOrbital}. Clearly, verifying stability properties for such systems in the full nonlinear setting would be beneficial. However, due to the complex dynamics, stability analysis so far has been limited to local stability analysis using linearization \cite{mostazaMethodologyAnalyzeAttitude2016} or Monte-Carlo methods \cite{livadiottiUncertaintiesDesignActive2022}.

A promising method to overcome these limitations and provide systematic analysis tools is sum-of-squares (SOS) programming. SOS formulates the search for a polynomial Lyapunov or storage function as a convex semidefinite program (SDP), which can be solved efficiently \cite{papachristodoulouTutorialSumSquares2005,parriloSemidefiniteProgrammingRelaxations2003a}.
Given its success in various robust control and even synthesis tasks \cite{ebenbauerAnalysisDesignPolynomial2006,jarvis-wloszekControlsApplicationsSum2003}, the systematic application of SOS to constrained spacecraft attitude systems holds significant potential for practical application.
While SOS is powerful, a sophisticated framework for applying it directly to systems on non-contractible manifolds, particularly for global stability properties, is still missing. Initial steps have been taken, for example, to verify local stability of quaternion-based spacecraft attitude control systems using SOS \cite{misraSumofSquaresBasedComputation2020}. However, a general methodology is needed.

This paper presents a novel framework to close this gap, enabling the systematic stability analysis of constrained nonlinear attitude control systems using SOS programming. The main contributions of this paper are:
\begin{enumerate}
	\item The derivation of a general dissipativity condition to verify almost global asymptotic stability for systems evolving on non-contractible manifolds.
	\item A systematic method to relax this theoretical condition into a convex sum-of-squares program that can be solved efficiently using standard SDP solvers.
	\item The application of this framework to two relevant spacecraft attitude stabilization problems, including a complex VLEO case with aerodynamic torques, providing the first systematic SOS-based almost global analysis for such a system.
\end{enumerate}

The remainder of this paper is structured as follows. \Cref{sec:prel} introduces the formal problem statement and relevant definitions. In \Cref{sec:methodology}, we present our core theoretical contribution: the dissipativity condition for AGAS on manifolds. This condition is then relaxed to a tractable SOS program in \Cref{sec:sos}. Finally, \Cref{sec:examples} demonstrates the efficacy of our method through two relevant spacecraft attitude control case studies.

	\section{Preliminaries}
	\label{sec:prel}

	\subsubsection*{Notation}
	Let $\R$ be the set of real numbers, $\R^n$ the $n$-dimensional Euclidean vector space, and $\R^{m\times m}$ the (vector) space of $m \times m$ real matrices. Furthermore, $\mathbb{S}_m^{+}$ denotes the set of positive semidefinite matrices in $\mathbb{R}^{m \times m}$.
	 Let $\norm{\cdot}$ and $\inner{\cdot}{\cdot}$ denote the Euclidean norm and inner product, respectively.
	For a vector-valued function $\phi: \mathbb R^n \to \mathbb R^m$, the derivative of $\phi$ at $\bar x \in \mathbb R^n$ is the linear operator $\nabla \phi(\bar x): \mathbb R^n \to \mathbb R^m$ defined as
	
	\begin{equation*}
	\nabla \phi(\bar x): x \mapsto J_{\bar x} x, \quad [J_{\bar x}]_{ij} = \frac{\partial \phi_i}{\partial x_j}(\bar x),
	\end{equation*}
    where $J_{\bar x} \in \mathbb R^{m \times n}$ is the Jacobian matrix of $\phi$ at $\bar x$.

	\subsection{Problem Statement}
\label{sec:problem_statement}

Consider the autonomous nonlinear system
\begin{align}
	\dot{x} &= f(x), \quad x(0) \in \setH, \quad x(t) \in \setH \ \text{for all} \ t \geq 0, \label{eq:dyn_system}
\end{align}
where $x \in \mathbb R^n$ is the system state, $f: \mathbb R^n \to \mathbb R^n$ is a smooth vector field and 
\[
\setH := \left\{ x \in \R^n \ \middle|\ h(x) = 0 \right\}
\]
with $h: \mathbb R^n \to \mathbb R$ is a smooth and non-contractible manifold of dimension $\ell$. 
We assume that the system has a finite number of isolated equilibrium points $x^*_j \in \setH$ satisfying
\begin{align}
	f(x^*_j) = 0, \quad j = 1,\ldots,k, \label{eq:equilibria}
\end{align}
of which, without loss of generality, $x^*_1=0$ will be the point of interest.

	\begin{assumption}
		\label{ass:positive_real_parts}
	For $j = 2, \ldots, k$, it holds that
	\begin{equation}
        \label{eq:positive_real_parts}
		\exists v \in \operatorname{ker} \left(\nabla h(x_j^*) \right), \exists \lambda > 0 \text{ s.t. } \left[J_{x^*_j} - \lambda I \right] v = 0.
	\end{equation}
	%that the Jacobian $J_{x^*_j}$ has at least one eigenvalue with positive real part.
	\end{assumption}

Eq.~\eqref{eq:positive_real_parts} corresponds to an unstable eigendynamic (locally around $x_j^*$) of \eqref{eq:dyn_system} within the tangential subspace of $\mathcal H$. If $H \in \mathbb R^{n \times \ell}$ is a basis of the tangential space of $\mathcal H$ at $x_j^*$, then \eqref{eq:positive_real_parts} is equivalent to the matrix $H^\top J_{x_j^*} H$ having a positive eigenvalue.
\begin{remark}
	While Assumption \ref{ass:positive_real_parts} excludes unstable equilibria without positive eigenvalues, it is rarely limiting in practical mechanical applications such as spacecraft attitude control. Therefore, we will not relax this condition within the scope of this work.
\end{remark}

	\subsection{Stability Definitions}
	To formulate stability properties on the manifold $\setH$, we first recall some standard definitions.

	\begin{definition}
	A set $\Omega \subseteq \setH$ is  \emph{positively invariant} if and only if for all $x(0) \in \Omega$, the corresponding solution $x(t)$ of~\eqref{eq:dyn_system} satisfies $x(t) \in \Omega$ for all $t \geq 0$.
	\end{definition}

	We now define the local stability properties of an equilibrium point $x^*$ on $\setH$.

	\begin{definition}
	An equilibrium point $x^* \in \mathcal H$ of~\eqref{eq:dyn_system} is  \emph{Lyapunov stable} if and only if for all $\epsilon > 0$ there exists a $\delta > 0$ such that, for all $x(0) \in \setH$ with $\norm{x(0) - x^*} < \delta$, it holds that $\norm{x(t) - x^*} < \epsilon$ for all $t \geq 0$.  
	If, in addition, there exists a $\delta > 0$ such that, for all $x(0) \in \setH$ with $\norm{x(0) - x^*} < \delta$, it holds that $\lim_{t \to \infty} \| x(t) - x^* \|_2 = 0$, then $x^*$ is \emph{locally asymptotically stable}.
	\end{definition}

	To provide stronger stability guarantees, we further consider the following property.

	\begin{definition}
    An equilibrium point $x^* \in \mathcal{H}$ of~\eqref{eq:dyn_system} is \emph{almost globally asymptotically stable} if $x^*$ is Lyapunov stable and, for almost all\footnote{That is, the complementary set of initial conditions has measure zero in $\mathcal{H}$. For a formal definition of sets of measure zero on smooth manifolds, we refer the reader to \cite[Prop. 6.5]{leeIntroductionSmoothManifolds2012}.} $x(0) \in \mathcal{H}$, it holds that $\lim_{t \to \infty} \| x(t) - x^* \|_2 = 0$.
\end{definition}
	
	\subsection{Existing Results}
	\label{sec:previous_work}
	Systems evolving on manifolds pose particular challenges for stability analysis, especially when the manifold is not contractible.  
	From Sontag’s condition, the following negative result can be obtained.

	\begin{proposition}\cite[Corollary~5.9.13]{sontagMathematicalControlTheory1998}
		\label{prop:non_contractible}
	Systems of the form $\eqref{eq:dyn_system}$ do not admit globally asymptotically stable equilibria.
	\end{proposition}

	Consequently, for systems of the form~\eqref{eq:dyn_system}, no Lyapunov function can establish global asymptotic stability of $x_1^*$.  
	It is therefore of interest to study stability properties that are as strong as possible under these topological constraints.  
	In particular, our analysis is based on stability properties of the remaining equilibria.
	
	\section{Stability on Manifolds}
	\label{sec:methodology}

    This section provides the first result of the paper, which is a dissipativity condition to verify almost global asymptotic stability specifically tailored for SOS programming. This condition will then be relaxed to a SOS program in the following section.

    \begin{theorem}[Asymptotic Stability on $\setH$]
	\label{thm:agas}
If there exists a continuously differentiable function $V: \setH \rightarrow \R$ and a scalar $\epsilon > 0$ such that
\begin{align}
	V(0) &= 0 \label{eq:V_zero}\\
	V(x) &> 0 \quad \text{for all } x \neq 0 \label{eq:V_pos_def} \\
	\langle \nabla V(x), f(x) \rangle &\leq -\epsilon \prod_{i=1}^{k} \norm{x-x_i^*}^2 \quad \text{for all } x \in \setH,\label{eq:Vdot_negative}
\end{align}
then $x^*_1 = 0$ is a locally asymptotically stable equilibrium of system (\ref{eq:dyn_system}). 
\end{theorem}

\begin{proof}
	Consider the set $\Omega_c := \{x \in \setH \mid V(x) \leq c\}$ with $c>0$. For a sufficiently small $c$, $\Omega_c$ is a compact, positively invariant neighborhood of the origin that excludes the equilibria $\{x_2^*, \dots, x_k^*\}$. Within this set, the derivative of $V$ along a solution is strictly negative for all non-zero states, since
	$$ \dot{V}(x) \leq -\epsilon \prod_{i=1}^{k} \norm{x-x_i^*}^2 < 0 $$
	for any $x \in \Omega_c \setminus \{x_1^*\}$.
	Therefore, by Lyapunov's direct method, the equilibrium $x_1^*$ is locally asymptotically stable.
\end{proof}
Although this property is useful, it only guarantees convergence for initial conditions in a neighborhood of the origin. Specifically in the context of spacecraft control systems, it is often desirable to establish stability properties for a larger set of initial conditions.
Given that global asymptotic stability of $x_1^*$ on $\setH$ cannot hold (Proposition \ref{prop:non_contractible}), we instead make a statement about trajectories converging to other equilibria by adapting \cite[Proposition~11]{monzonLocalGlobalAspects2006} to dynamic systems in manifolds.
%Although global asymptotic stability on $\setH$ is impossible according to Proposition \ref{prop:non_contractible}, the following proposition allows us to obtain \emph{almost global} asymptotic stability properties by adapting \cite[Proposition~11]{monzonLocalGlobalAspects2006} to dynamic systems in manifolds.
	\begin{proposition} 
		\label{prop:unstable_equilibria}
	Under Assumption \ref{ass:positive_real_parts}, the set of initial conditions $x_0 \in \setH$ such that
	\[
	\bigg\{x(0) \in \mathcal{H} \; \bigg|\; \bigvee_{j = 2,\ldots,k} \lim_{t \to \infty} \| x(t) - x^*_j \|_2 = 0 \bigg\}
	\]
    has measure zero on $\setH$.
	\end{proposition}

	\begin{proof}
		Since $\setH$ is a smooth manifold, the dynamics in \eqref{eq:dyn_system} linearized around any $x_j^*$, $j = 2, \ldots, k$, are restricted to the subspace $\mathcal{T}_{x_j^*}$ tangential to $\setH$ at $x_j^*$. This subspace can further be divided into stable, unstable, and center subspaces $E^{\mathrm{s}}_{\mathrm{loc}}, E^\mathrm{u}_{\mathrm{loc}}$ and $E^\mathrm{c}_{\mathrm{loc}}$ corresponding (locally) to the stable, unstable, and center manifold, respectively \cite[App. \,B.1]{isidoriNonlinearControlSystems1995}. Given Assumption \ref{ass:positive_real_parts}, the local unstable subspace $E^\mathrm{u}_{\mathrm{loc}}$ (of $\mathcal T_{x_j^*}$) has dimension greater than zero, which implies that the dimension of the center stable manifold, corresponding to $E^\mathrm{s}_{\mathrm{loc}} \oplus E^\mathrm{c}_{\mathrm{loc}}$, is strictly smaller than the dimension of $\mathcal{T}_{x_j^*}$ and therefore has zero Lebesgue measure.
		This local result can be extended to the entire region of attraction of $x_j^*$ on $\setH$. The argument for this extension is analogous to the proof given in \cite[Proposition~11]{monzonLocalGlobalAspects2006}.
		Thus, the set of initial conditions with trajectories converging to $x_j^*$ has measure zero in $\setH$, as does the union over $j = 2,\ldots,k$.
	\end{proof}

Now we can establish almost global asymptotic stability under additional conditions on the Lyapunov function.
The intuition behind this approach is simple, yet useful. The supply rate term on the right-hand side of \eqref{eq:Vdot_negative} ensures the time derivative of $V$ to be strictly negative everywhere on $\setH$ except for the equilibria, allowing for the application of LaSalle's Theorem.

\begin{theorem}[Almost Global Stability]\label{thm:almost_global_stability}
	If there exists a continuously differentiable function $V: \setH \rightarrow \R$ which, in addition to \eqref{eq:V_zero}--\eqref{eq:Vdot_negative}, is radially unbounded\footnote{That is, for $x \in \setH, V(x) \rightarrow \infty$ as $\norm{x} \rightarrow \infty$.}, then the equilibrium point $x^*_1 = 0$ is almost globally asymptotically stable.
\end{theorem}

\begin{proof}
	% Note: radially unbounded means that every sublevel set is a compact set!
	Consider the set $\Omega_c := \left\{ x \in \setH \mid V(x) \leq c \right\}$ for some $c > 0$. Since $V$ is continuous, $\Omega_c$ is closed for all $c > 0$. Furthermore, $\Omega_c$ is bounded since $V$ is radially unbounded on $\setH$. Therefore, $\Omega_c$ is compact. Given that $\Omega_c$ is positively invariant for all $c>0$, it follows that all solutions starting in $\Omega_c$ remain in $\Omega_c$ for all time. Furthermore, the largest invariant set in $\Omega_c$ where $\inner{\nabla V(x)}{f(x)} = 0$ is given by $\mathcal{E} = \left\{ x^*_1, x^*_2, \ldots, x^*_k \right\}$ since \eqref{eq:Vdot_negative} holds. According to LaSalle's Invariance Principle \cite[Theorem~1]{lasalleExtensionsLiapunovsSecond1960}, all trajectories starting in $\Omega_c$ converge to $\mathcal{E}$. Since this holds for all $c>0$, $\mathcal E$ is attractive on $\setH$. 
	Under Assumption \ref{ass:positive_real_parts}, the set of initial conditions for which the solution converges to $x^*_j$ with $j=2,\ldots,k$ has measure zero on $\setH$ according to Proposition \ref{prop:unstable_equilibria}. Hence, almost all initial conditions in $\setH$ converge to $x^*_1 = 0$.
\end{proof}

Crucially, Theorems \ref{thm:agas} and \ref{thm:almost_global_stability} are formulated to be directly verifiable using SOS programming, as shown in the next section. As a result, we provide a systematic and computationally tractable framework for the analysis of systems including, for example, those arising in spacecraft attitude control.

\section{Sum-Of-Squares Optimization}
\label{sec:sos}
In this section, we briefly introduce the SOS programming method and present our second contribution by relaxing the problem defined in Theorems \ref{thm:agas} and \ref{thm:almost_global_stability} to an SOS program.

A polynomial up to degree $d\in \mathbb{N}_0$ is a sum of monomials
\begin{equation}
    \pi = \sum_{\|\alpha\|_1\ \leq d} c_{\alpha}x^{\alpha}, \quad x^\alpha := x_1^{\alpha_1} \dots x_\ell^{\alpha_\ell},
\end{equation}
where $x = (x_1, \dots, x_\ell)$ is a tuple of $\ell \in \mathbb N$ indeterminate variables, $\left\{ c_\alpha \right\}_\alpha \subset \mathbb R$ are the coefficients of $\pi$ and $\alpha = (\alpha_1, \dots, \alpha_\ell)\in \mathbb N_0^\ell$ are the degrees. The set of polynomials in $x$ with real coefficients up to degree $d$ is denoted by $\mathbb R_d[x]$

\begin{definition}
	\label{def:sos}
	A polynomial $\pi \in \R_{2d}[x]$ is \textit{sum-of-squares} ($\pi \in \Sigma_{2d}[x]$) if and only if there exist $m \in \mathbb{N}$ and polynomials $p_i(x) \in \R_d[x]$ such that $\pi(x) \equiv \sum_{i=1}^m p_i(x)^2$.
\end{definition}

To find such an SOS decomposition, we can use the following proposition \cite{parriloSemidefiniteProgrammingRelaxations2003a}.
\begin{proposition}
	A polynomial $\pi \in \mathbb R_{2d}[x]$ is a \emph{sum-of-squares polynomial} if and only if there exists a matrix $Q\in \mathbb S^+$ and a vector of monomials $z(x)$ of degree $\leq d$ such that
\begin{equation}
    \label{eq:zQz}
    \pi = z(x)^{\top} Q z(x).
\end{equation}
\end{proposition}
Expanding the right-hand side of \eqref{eq:zQz} and comparing the coefficients with those of $\pi$ leads to a set of linear equality constraints on the entries of $Q$.
Since $\pi$ being \emph{sum-of-squares} is equivalent to $Q$ being positive semidefinite, finding such a (not necessarily unique) $Q$ can be cast into a semidefinite program allowing us to use efficient numerical solvers to find a suitable SOS decomposition (see \cite{parriloSemidefiniteProgrammingRelaxations2003a} for details on the semidefinite program formulation).
Since translating an SOS program into an SDP can be cumbersome, tools such as CaΣoS \cite{Cunis2025acc} simplify the process of converting SOS programs to the according SDP formulation for numerical solution. We use CaΣoS (v1.0.0) in this work.

Under the additional assumption that $f \in \realpoly{d_{f}}$ and $h \in \realpoly{d_{h}}$ are polynomials, we can relax the problem defined in Theorems \ref{thm:agas} and \ref{thm:almost_global_stability} to a convex SOS program with
\begin{subequations}
\label{eq:SOS_program}
\begin{align}
	\text{find} \quad& V(x) \in \sos{d_{V}}, \ p(x) \in \realpoly{d_{p}} \\
	\text{s.t. } \quad
		& V(x) - \epsilon_1 \norm{x}^2 \in \sos{d_{V}}, \label{eq:sos_pos_def}\\
		\text{and } \quad & -\inner{\nabla V(x)}{f(x)} + p(x) h(x) \notag \\
		& \qquad - \epsilon_2 \prod_{i=1}^{k} \norm{x - x_i^*}^2 \in \sos{2d},
		\label{eq:sos_vdot_neg}
\end{align}
\end{subequations}
where $V(x)$ and $p(x)$ are polynomial decision variables and $\epsilon_i>0$ are small positive constants. The degrees $d_{v}, d_{p} \in \mathbb{N}$ are chosen \textit{a priori} such that $2d \geq \max{} \{ d_{f}+d_{V}-1, d_{p}+d_{h}, 2k \}$.
Constraint \eqref{eq:sos_vdot_neg} uses the technique reminiscent of the well-known S-procedure described in \cite{papachristodoulouTutorialSumSquares2005} to ensure that the dissipativity condition \eqref{eq:Vdot_negative} holds on $\setH$.

\section{Examples}
\label{sec:examples}
In this section, we demonstrate the utility of our theoretical results and the derived SOS relaxation by applying our framework to two separate spacecraft attitude stabilization problems. In both cases, we compute Lyapunov functions certifying almost global asymptotic stability using sum-of-squares programming.
\begin{itemize}
    \item In Example 1, we use a global polynomial fit for nonlinear aerodynamics to find a suitable Lyapunov function to certify aerodynamic stability of a feathered satellite design under rate damping.
    \item In Example 2, we demonstrate the method's feasibility for higher-dimensional systems (7 states), applying it to attitude dynamics using quaternions which are essential in the field.
\end{itemize}

\subsubsection*{Coordinate Frames and Notation}

In the following examples, `$\mathrm{B}$' denotes the body frame aligned with the satellite's principal axes of inertia (see Fig.~\ref{fig:satellite}), and `$\mathrm{I}$' denotes an inertial frame. In the second example, the orbital frame is denoted by `$\mathrm{O}$', with its $\hat{z}_{\mathrm{O}}$-axis pointing to the Earth's center, its $\hat{x}_{\mathrm{O}}$-axis pointing in the direction of orbital velocity and its $\hat{y}_{\mathrm{O}}$-axis completing the right-handed coordinate system.
The normalized vector $\hat{v}$ is the unit vector pointing in the direction of a vector $v$. The skew-symmetric matrix associated with the cross product of a vector $a \in \R^3$ is denoted by $\skewop{a} \in \R^{3 \times 3}$.
For attitude parametrizations, $(\cdot)_{\mathrm{BA}}$ denotes the attitude of frame $\mathrm{B}$ with respect to frame $\mathrm{A}$. Furthermore, $\omega_{\mathrm{BA}}^{\mathrm{C}}$ denotes the angular velocity of frame $\mathrm{B}$ with respect to frame $\mathrm{A}$ expressed in frame $\mathrm{C}$.
\subsection{Aerostability of Satellites}

	%%%%%%%%%%%%%%%%%%%%%%%%%%%%%%%%%%%%%%%%%%%%%%%%%%%%%%%%%%%%%%%%%%%%%%
	% --------------------------------------------------------------------
	%%%%%%%%%%%%%%%%%%%%%%%%%%%%%%%%%%%%%%%%%%%%%%%%%%%%%%%%%%%%%%%%%%%%%%
	We examine the \textit{aerostability} of the SOAR satellite \cite{crispSOARSatelliteOrbital}, a 3U cubesat with feathered geometry operating in VLEO. Its parameters are summarized in Table~\ref{tab:satellite_params}. At this altitude, aerodynamic torques are the dominant disturbance acting on the satellite attitude dynamics. The idea behind this approach is to utilize aerodynamic torques to passively align the satellite's body frame with the incoming flow direction under active rate damping. Using our methodology, we provide a formal verification of the aerostability property in the almost-global sense.
	
	% table of satellite parameters (see figures/satellite_parameters_table.tex)
	\begin{table}[ht]
\centering
\resizebox{\columnwidth}{!}{%
\begin{tabular}{lll}
Parameter & Symbol [unit] & Value \\ 
\hline
Inertia  & $\Theta$ [kg\,m$^2$] & $\mathrm{diag}(0.0288, 0.0392, 0.0392)$ \\
%Principal moment (x-axis) & $\Theta_x$ [kg\,m$^2$] & $0.0288$ \\ 
%Principal moment (y-axis) & $\Theta_y$ [kg\,m$^2$] & $0.0392$ \\ 
%Principal moment (z-axis) & $\Theta_z$ [kg\,m$^2$] & $0.0392$ \\ 
Wing panel area (one side) & $A$ [m$^2$] & $0.0342$ \\ 
Wing CoP x-position & $x_{\mathrm{cp}}$ [m] & $-0.1310$ \\ 
Wing CoP distance to $x_{\text{B}}$ axis & $r_{\perp}$ [m] & $0.3350$ \\ 
\end{tabular}%
}

\caption{Satellite parameters.}
\label{tab:satellite_params}
\end{table}

	\begin{figure}[ht]
		\centering
		\includegraphics[width=0.55\columnwidth]{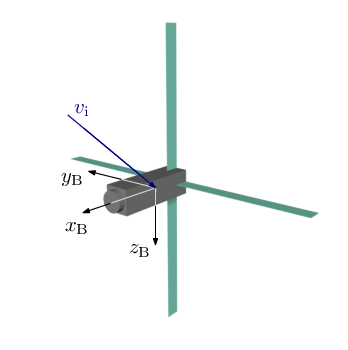}
		\caption{Satellite geometry and corresponding body frame.}
		\label{fig:satellite}
	\end{figure}
	
	\subsubsection{Attitude Kinematics and Dynamics}
	Consider the two-axis attitude kinematics and dynamics \cite[Ch.~5.2.2]{fichterPrinciplesSpacecraftControl2023} of a satellite in VLEO subject to aerodynamic torques given by
    \begin{subequations}
	\begin{align}
		\wdot &= \w\times \om  \label{eq:pointing_kinematics_example} \\
		\dotom &= -\Theta^{-1} \left[\om \times \Theta \om +  \tau_{\mathrm{aero}}^{\mathrm{(B)}} + \tau_{\mathrm{c}}^{\mathrm{(B)}}\right] \label{eq:pointing_dynamics_example}
	\end{align}
    \end{subequations}
	where $\w$ is the direction vector defined as the negative normalized incoming wind velocity $-\hat{v}_{\mathrm{i}}$, $\om$ is the rotational velocity and $\Theta \in \R^{3\times 3}$ is the inertia matrix expressed in the body frame $\mathrm{B}$. Furthermore, $\tau_{\mathrm{aero}}^{\mathrm{(B)}}$ and $\tau_{\mathrm{c}}^{\mathrm{(B)}}$ are the aerodynamic and control torques expressed in the body frame $\mathrm{B}$. Since $\w$ is a unit direction vector, the system dynamics evolve on the manifold $\setH = \left\{ (\w, \om) \in \R^6 \mid \|\w \|_2 = 1 \right\}$. 

	\subsubsection{Aerodynamic Torque}
	The aerodynamic torque is computed as the sum of the individual panel contributions, expressed directly as
	\begin{equation}
		\tau_{\mathrm{aero}}^{\mathrm{(B)}} = \sum_{j=1}^{p} r_j \times \left( \frac{\rho}{2} \norm{v_{\mathrm{i}}}^2 A_j \left[ H_1 \hat{n}_j + H_2 \hat{v}_{\mathrm{i}} \right] \right)
		\label{eq:aero_torque_combined}
	\end{equation}
	where $r_j$ is the position vector from the satellite’s center of mass to the center of pressure of panel $j$. 
	The aerodynamic force on each panel is modeled as a two-sided flat plate according to Sentman’s method~\cite{sentman_free_1961}, with $\rho$ denoting the atmospheric density, $A_j$ the panel area, $\hat{n}_j$ the panel normal vector. The expressions $H_1$ and $H_2$ are scalar terms containing all nonpolynomial dependencies. Their full expressions are given in the appendix.
	All environmental parameters required by the model are chosen as reasonable constant parameters at an altitude of $300 \mathrm{ km}$ (see \Cref{tab:envparams}). 
	% Input table with environmental parameters
	\begin{table}[ht]
		\centering
		% Auto-generated by export_environment_to_tex.m on 23-Oct-2025 10:38:03
\begin{tabular}{lll}
Parameter & Symbol [unit] & Value \\ 
\hline
Altitude & $h$ [km] & $300$ \\ 
Density & $\rho$ [kg/m$^3$] & $1.88\,\times\,10^{-11}$ \\ 
Orbital speed & $V_{\text{i}}$ [m/s] & $7725.84$ \\ 
Energy accommodation coefficient & $\alpha_{\text{E}}$ [-] & $0.95$ \\ 
Thermal speed ratio & $s_{\text{i}}$ [-] & $7.86$ \\ 
\end{tabular}

		\caption{Environment parameters at $h=300\,\mathrm{km}$.}
		\label{tab:envparams}
	\end{table}

	Given those parameters, $H_1$ and $H_2$ can be expressed as functions of $\cos \delta = -\inner{\hat{v}_{\mathrm{i}}}{\hat{n}}$ which is the angle between the incoming wind direction and the panel normal vector.
	To obtain polynomial dynamics, we coarsely approximate $H_1$ and $H_2$ using a least-squares polynomial fit of the degrees 2 and 4, respectively, over the interval $\cos \delta \in [-1,1]$ as shown in \Cref{fig:poly_fit}. 

	% Polyfit figure as .tex file with figurewidth
	\begin{figure}[ht]
		\centering
		\resizebox{\figurewidth}{!}{\input{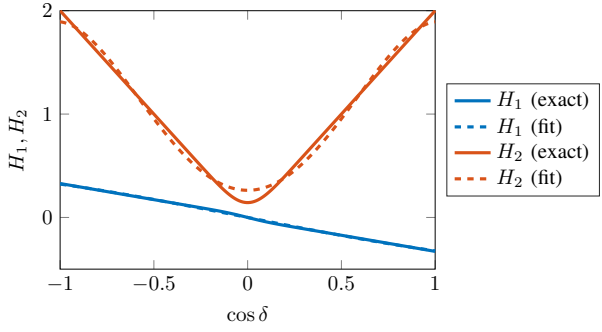}}
		\caption{Polynomial fit of $H_1$ and $H_2$ over the interval $\cos \delta \in [-1,1]$.}
		\label{fig:poly_fit}
	\end{figure}

	\subsubsection{Closed-Loop System Dynamics}
	Given that the rotational velocity compared to the orbital velocity is insignificant, we assume that the aerodynamic torque only depends on the satellite orientation $\w$. Consequently, active damping is required to ensure asymptotic stability of the equilibrium point $\bar{x}_1^*$. The control torque is then simply defined as
	\begin{equation}
		\tau_{\mathrm{c}}^{\mathrm{(B)}} = -k_\mathrm{D} \Theta \om
	\end{equation}
	where $k_\mathrm{D} > 0$ is a constant gain.
	To obtain the final closed-loop system dynamics, we define the original state vector as
		$\bar{x} = (\w, \om) \in \R^6$.
	Due to the satellite's symmetry, there are two isolated equilibrium points, namely, 
	\begin{equation}
		\bar{x}^*_1 = \begin{bmatrix}
			1 \\ \mathbf{0}_5
		\end{bmatrix}, \quad
		\bar{x}^*_2 = \begin{bmatrix}
			-1 \\ \mathbf{0}_5
		\end{bmatrix}.
	\end{equation}

	To obtain the final closed-loop dynamics in the form of \eqref{eq:dyn_system}, we define a state transformation
	\begin{equation}
		x = S(\bar{x} - \bar{x}^*_1) \in \R^6. \label{eq:shift_and_scale}
	\end{equation}
	by shifting the origin into the first equilibrium point and scaling the state variables by a diagonal matrix $S = \mathrm{diag}(1,1,1,20,20,20)$. The scaling is heuristically determined to improve the numerical conditioning of the SOS program by keeping the polynomial coefficients of the transformed system at comparable magnitudes.
	The transformed system dynamics can be found in the appendix. The Jacobian $H^{\top}\nabla f(x^*_2)H$ has two eigenvalues with positive real part, satisfying Assumption \ref{ass:positive_real_parts}.

	\subsubsection{Stability Analysis}
	For $k_{\mathrm{D}}=0.016$ and $\epsilon_1 = \epsilon_2 = 1\times 10^{-5}$, we solve the SOS program \eqref{eq:SOS_program} using CaΣoS \cite{Cunis2025acc} with MOSEK \cite{mosek} as the underlying solver. Using polynomial decision variables V(x) and p(x) of degrees 4 and 6, respectively, we obtain a suitable quartic Lyapunov function (omitted for brevity). Its evolution and the corresponding system trajectories are illustrated for an exemplary solution in \Cref{fig:V_along_trajectory}
	
	% Figure of V along trajectory
	\begin{figure}[ht]
		\centering
		\resizebox{\figurewidth}{!}{\input{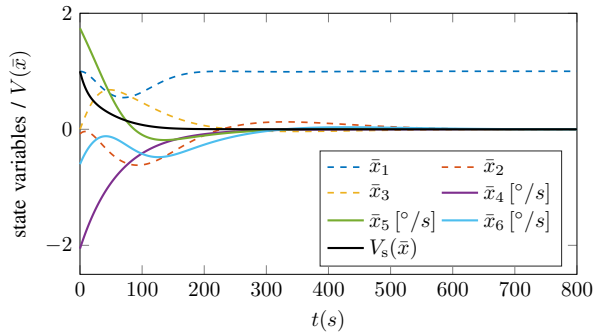}}
		\caption{Example 1: Value of the scaled Lyapunov function ${V_{\mathrm{s}}(\bar{x})}$ and the original coordinates along an exemplary solution trajectory. Body rates are shown in $\mathrm{deg}/\mathrm{s}$ for better readability.}
		\label{fig:V_along_trajectory}
	\end{figure}

	%%%%%%%%%%%%%%%%%%%%%%%%%%%%%%%%%%%%%%%%%%%%%%%%%%%%%%%%%%%%%%%%%%%%%%
	% --------------------------------------------------------------------
	%%%%%%%%%%%%%%%%%%%%%%%%%%%%%%%%%%%%%%%%%%%%%%%%%%%%%%%%%%%%%%%%%%%%%%
	\subsection{Quaternion-Based Stabilization of Satellites subject to Gravity Gradient Torque}
	% --------------------------------------------------------------------
	%%%%%%%%%%%%%%%%%%%%%%%%%%%%%%%%%%%%%%%%%%%%%%%%%%%%%%%%%%%%%%%%%%%%%%
	At higher altitudes, the influence of aerodynamic torque becomes negligible and gravity gradient torque becomes the dominant environmental disturbance. In this example, we consider the three-axis attitude stabilization of the same spacecraft in circular orbit subject to gravity gradient torque using quaternion feedback control at an altitude of 500 km.
	
	\subsubsection{Kinematics and Dynamics}
	The attitude kinematics and dynamics of a rigid body spacecraft in a circular orbit \cite{chaturvediAlmostGlobalAttitude2006} are given by
    \begin{subequations}
	\begin{align}
		\qBOdot &= \frac{1}{2} \begin{bmatrix}
			0 & -(\omBO)^{\top} \\
			\omBO & -\skewop{\omBO} \\
		\end{bmatrix} \qBO \label{eq:ex2_quat_kinematics} \\
		\omBOdot &= \Theta^{-1} 
		\left[
		-\left(\omBO-\omega_0\,\yBO\right) \times \Theta
		\left(\omBO-\omega_0\,\yBO\right) \right. \notag \\
		& \quad \left. -\omega_0 \Theta \omBO \times \yBO + \tau_{\mathrm{c}}^{\mathrm{(B)}} + \tau_{\mathrm{gg}}^{\mathrm{(B)}} \right]\label{eq:ex2_dynamics}
	\end{align}
    \end{subequations}
	where $\qBO=\begin{bmatrix}
		q_0 & q_v^{\top}
	\end{bmatrix}^{\top} \in \R^4$ is the unit quaternion representing the attitude of the body frame $\mathrm{B}$ with respect to the orbital frame $\mathrm{O}$, and $\omega_0$ is the scalar orbital rate. 
    The system dynamics evolve on the manifold $\setH = \left\{ (\qBO , \omBO) \in \R^4 \times \R^3 \mid \|\qBO \|_2 = 1 \right\}$. 
    We define the unit basis vectors of the orbital frame expressed in the body frame as $\xBO$, $\yBO$, and $\zBO$; those form the columns of the direction cosine matrix
	\begin{equation*}
		T_{\mathrm{OB}}=
        (2q_0^2-1) {I}_{3\times 3} + 2(q_vq_v^{\top} - q_0[q_v \times]).
        %2\begin{bmatrix}
		%	q_0^2 + q_1^2 - \frac{1}{2} & q_1 q_2 + q_0 q_3 & q_1 q_3 - q_0 q_2 \\
		%	q_1 q_2 - q_0 q_3 & q_0^2 + q_2^2 - \frac{1}{2} & q_2 q_3 + q_0 q_1 \\
		%	q_1 q_3 + q_0 q_2 & q_2 q_3 - q_0 q_1 & q_0^2 + q_3^2 - \frac{1}{2}
	    %\end{bmatrix}
	\end{equation*}

	Furthermore, the gravity gradient torque in circular orbits is given by \cite[Ch.\,1.2.2]{fichterPrinciplesSpacecraftControl2023}
	\begin{equation*}
		\tau_{\mathrm{gg}}^{\mathrm{(B)}} = 3 \omega_0^2 \left( \zBO \times \Theta \zBO \right).
	\end{equation*}

	 \subsubsection{Control Design}
	The goal of this example is to stabilize the spacecraft attitude by aligning the body frame $\mathrm{B}$ with the orbital frame $\mathrm{O}$. To achieve this, we employ a proportional-derivative (PD) control law based on \cite{wieQuaternionFeedbackRegulator1989} defined as
	\begin{equation}
		\tau_{\mathrm{c}}^{\mathrm{(B)}} = -k_\mathrm{p}\Theta q_v - k_\mathrm{d}\Theta \omBO
	\end{equation}
	where $k_\mathrm{p}, k_\mathrm{d} > 0$ are constant scalar gains.
	\subsubsection{Closed-Loop System Dynamics}
	Similar to the first example, we define the state vector as $\bar{x} = (\qBO, \omBO) \in \R^7$.
	The resulting closed-loop system exhibits the two isolated equilibrium points 
	\begin{equation}
	\bar{x}^*_1 = \begin{bmatrix}
			1 \\ \mathbf{0}_{6 \times 1}
			\end{bmatrix}, \quad
		\bar{x}^*_2 = \begin{bmatrix}
			-1 \\ \mathbf{0}_{6 \times 1}
			\end{bmatrix}.
	\end{equation}
	Again, we shift and scale the system dynamics with (\ref{eq:shift_and_scale}) using the scaling matrix $S$ (see \Cref{tab:params_example2}). Computing the eigenvalues of $H^{\top}\nabla f(x^*_2)H$ yields two positive real eigenvalues satisfying Assumption \ref{ass:positive_real_parts}.
	% Input table with environmental parameters
	\begin{table}[ht]
		\centering
		\begin{tabular}{lll}
			Parameter & Symbol [unit] & Value \\ 
			\hline
			Inertia Matrix & $\Theta$ [kg\,m$^2$] & see Table~\ref{tab:satellite_params}. \\ 
			Proportional gain & $k_\mathrm{p}$ & $0.0064$ \\
			Derivative gain & $k_\mathrm{d}$ & $0.08$ \\
			\hline
			Altitude & $h$ [km] & $500$ \\
			Orbital rate & $\omega_0$ [rad/s] & $0.0011$ \\
			\hline
			Scaling matrix & $S$ & $\mathrm{diag}(1,1,1,1,15,15,15)$ \\
			Small constants & $\epsilon_1, \epsilon_2$ & $1\times 10^{-5}$ \\
			\hline
		\end{tabular}
		\caption{Parameters for Example 2.}
		\label{tab:params_example2}
	\end{table}

	\subsubsection{Stability Analysis}
	The parameters used to apply our methodology are summarized in \Cref{tab:params_example2}. 
	Again, we solve the  SOS program defined by ($\ref{eq:SOS_program}$) with 
	CaΣoS and the underlying solver MOSEK.
	Using polynomial decision variables $V(x)$ of degree 4 and $p(x)$ of degree 6, we find the Lyapunov function $V(x)$ (omitted for brevity) certifying almost global asymptotic stability of the equilibrium point $x^*_1$.
    Its value and the course of the original coordinates are shown for an exemplary solution trajectory in \Cref{fig:V_along_trajectory_example2}.
	
	% Figure of V along trajectory (example 2)
	\begin{figure}[ht]
		\centering
		\resizebox{\figurewidth}{!}{\input{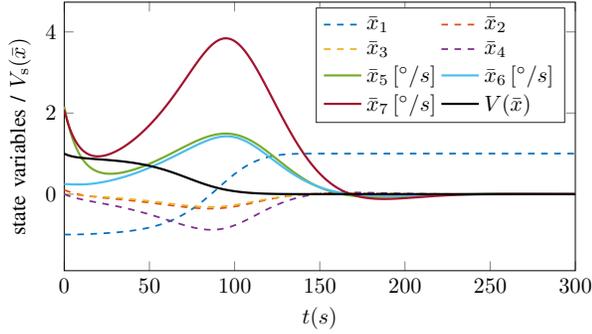}}
		\caption{Example 2: Value of the scaled Lyapunov function ${V_{\mathrm{s}}(\bar{x})}$ in original coordinates along a solution trajectory. Body rates are shown in $\mathrm{deg}/\mathrm{s}$ for better readability.}
		\label{fig:V_along_trajectory_example2}
	\end{figure}

	This result shows that our methodology can be applied to systems with seven state variables, making it viable for a wide range of spacecraft attitude control problems.
	
	\section{Conclusion}
	\label{sec:conclusion}
	We present a framework to systematically verify almost global asymptotic stability of nonlinear control systems evolving on non-contractible manifolds using LaSalle's invariance principle and SOS optimization. This approach directly addresses the analytical challenges of constrained attitude representations, such as unit quaternions, where global asymptotic stability is topologically infeasible. We establish these strong stability guarantees by constructing a Lyapunov function whose time derivative along solutions is strictly negative everywhere on the manifold except for the isolated equilibrium points.
	
	The methodology's practical utility is demonstrated in two relevant spacecraft attitude control problems. First, we formally verify the passive aerostability of a cubesat in VLEO under active rate damping, achieving, to our knowledge, the first quasi-global certification using Lyapunov's methods for this scenario. Second, we verify the almost global stability of a three-axis attitude acquisition system using quaternion feedback control subject to gravity gradient torques. In both cases, we successfully find suitable Lyapunov functions using SOS optimization, highlighting the framework's capacity to handle complex polynomial systems.

    While the current work addresses nominal system dynamics, the SOS-based methodology provides a clear path for extension. A critical next step is to extend this framework to address robust stability, providing formal guarantees for systems subject to the parametric or model uncertainty common in environments like VLEO.
    To facilitate this and encourage further development, the code used to generate the results in this paper is available as an open-source repository \cite{geyer-Supplementary25}.
	
	%%%%%%%%%%%%%%%%%%%%%%%%%%%%%%%%%%%%%%%%%%%%%%%%%%%%%%%%%%%%%%%%%%%%%%

\addtolength{\textheight}{-6cm}   % This command serves to balance the column lengths

%%%%%%%%%%%%%%%%%%%%%%%%%%%%%%%%%%%%%%%%%%%%%%%%%%%%%%%%%%%%%%%%%%%%%%%%%%%%%%%%

%%%%%%%%%%%%%%%%%%%%%%%%%%%%%%%%%%%%%%%%%%%%%%%%%%%%%%%%%%%%%%%%%%%%%%%%%%%%%%%%

%%%%%%%%%%%%%%%%%%%%%%%%%%%%%%%%%%%%%%%%%%%%%%%%%%%%%%%%%%%%%%%%%%%%%%%%%%%%%%%%
\section*{Appendix}

The scalar functions $H_1$ and $H_2$ used in (\ref{eq:aero_torque_combined}) are given by \cite{sentman_free_1961}
\begin{align} 
	H_{1}(\cosd) &= -\frac{1}{2 s_{\mathrm{i}}^2} \exp(-s_{\mathrm{i}}^{2} \cos^2 \delta) T^{-}(\cosd) \notag \\
	&\quad- \operatorname{erf}(s_{\mathrm{i}} \cosd) \left( \frac{1}{s_{\mathrm{i}}^{2}} + \frac{\sqrt{\pi} \cosd}{2s_{\mathrm{i}}} T^{-}(\cosd) \right) \notag \\
	&\quad- \frac{\sqrt{\pi} \cosd}{2s_{\mathrm{i}}} T^{+}(\cosd)
\\
		H_{2}(\cosd) &= \frac{2}{\sqrt{\pi}s_{\mathrm{i}}} \exp(-s_{\mathrm{i}}^{2} \cosd^2) + 2 \cosd \operatorname{erf}(s_{\mathrm{i}} \cosd)
	\end{align}
	with
	\begin{equation}
		T^{\pm}(\cosd) = T_{\mathrm{rat}}(\cosd) \pm T_{\mathrm{rat}}(-\cosd)
	\end{equation}
	and the temperature ratio
	\begin{align}
		T_\mathrm{rat}(C) &= \alpha_{\mathrm{E}} \frac{2k_{\mathrm{B}}T_{\mathrm{w}}}{m_{\mathrm{T}}\norm{v_{\mathrm{i}}}^{2}} s_{\mathrm{i}}^{2} + (1-\alpha_{\mathrm{E}})\left[ 1+ \frac{s_{\mathrm{i}}^{2}}{2} + \right.\notag\\
		& \left. \frac{1}{4} \frac{s_{\mathrm{i}}C \sqrt{\pi} \operatorname{erfc}(-s_{\mathrm{i}} C)}{\exp(-s_{\mathrm{i}}^{2}C^2) + s_{\mathrm{i}} C \sqrt{\pi} \operatorname{erfc}(-s_{\mathrm{i}} C)} \right]
	\end{align}
	as derived in \cite{tuttasGeneralizedTreatmentEnergy2025}. Parameters used in this model are given in \Cref{tab:envparams}.

The transformed closed-loop system dynamics $\dot{x} = f(x)$ for example 1 are given by the polynomials
	\begin{align*}
		f_1 &= -0.05 x_3 x_5 + 0.05 x_2 x_6 \\
		f_2 &= -0.05 x_6 + 0.05 x_3 x_4 - 0.05 x_1 x_6 \\
		f_3 &= 0.05 x_5 - 0.05 x_2 x_4 + 0.05 x_1 x_5 \\
		f_4 &= -0.016 x_4 \\
		f_5 &= -0.00226 x_3 - 0.016 x_5 + 0.0133 x_4 x_6 - 0.00807 x_2^2 x_3 \notag \\
		    & \quad  - 0.00799 x_3^3 + 0.00389 x_2^4 x_3 + 0.00389 x_3^5 \\
		f_6 &= 0.00226 x_2 - 0.016 x_6 - 0.0133 x_4 x_5 + 0.00799 x_2^3\notag \\
		    & \quad  + 0.00807 x_2 x_3^2 - 0.00389 x_2^5 - 0.00389 x_2 x_3^4.
	\end{align*}
with the manifold constraint
\begin{equation*}
	h(x) = 2x_1 + x_1^2 + x_2^2 + x_3^2.
\end{equation*}

%%%%%%%%%%%%%%%%%%%%%%%%%%%%%%%%%%%%%%%%%%%%%%%%%%%%%%%%%%%%%%%%%%%%%%%%%%%%%%%%

\IEEEtriggeratref{13}
\bibliographystyle{IEEEtran}
\bibliography{references}

\end{document}